\documentclass[journal]{IEEEtran}

\usepackage{amsmath,amssymb,amsthm}
\usepackage{mathtools}
\usepackage{cite}
\usepackage{algorithm}
\usepackage{algorithmic}
\usepackage{booktabs}
\usepackage{xcolor}
\usepackage{url}
\usepackage{enumitem}
\usepackage{graphicx}
\graphicspath{{figures/}}

\newtheorem{theorem}{Theorem}
\newtheorem{lemma}{Lemma}
\newtheorem{proposition}{Proposition}
\newtheorem{corollary}{Corollary}
\newtheorem{definition}{Definition}
\newtheorem{remark}{Remark}
\newtheorem{assumption}{Assumption}
\newtheorem{problem}{Problem}

\DeclareMathOperator{\col}{col}

\DeclareMathOperator{\inte}{int}
\DeclareMathOperator{\dist}{dist}

\newcommand{\R}{\mathbb{R}}

\newcommand{\cW}{\mathcal{W}}
\newcommand{\cS}{\mathcal{S}}
\newcommand{\cV}{\mathcal{V}}
\newcommand{\cX}{\mathcal{X}}
\newcommand{\cU}{\mathcal{U}}
\newcommand{\cZ}{\mathcal{Z}}
\newcommand{\cA}{\mathcal{A}}
\newcommand{\cN}{\mathcal{N}}
\newcommand{\cD}{\mathcal{D}}

\newcommand{\cG}{\mathcal{G}}

\newcommand{\cM}{\mathcal{M}}

\allowdisplaybreaks

\begin{document}

\title{Robust Receding Horizon Games with Additive Uncertainty}

\author{Dinesh Patra, Tanish~Jain
        and~Ashish~R.~Hota%
\thanks{D. Patra and A. R. Hota are with the Department of Electrical Engineering, Indian Institute of Technology (IIT) Kharagpur, India. T. Jain was with the Department of Aerospace Engineering, IIT Kharagpur, India. Email: dinesh.patra912@kgpian.iitkgp.ac.in, jtanish120204@gmail.com, ahota@ee.iitkgp.ac.in.}}

\maketitle

\begin{abstract}
We study a receding horizon game in which multiple agents drive linear systems subject to additive disturbances, private state and input constraints, and shared coupling constraints. We propose a robust game-theoretic control framework that combines tube-based constraint tightening with a finite-horizon generalized Nash equilibrium problem (GNEP), equipped with a discrete algebraic Riccati equation (DARE)-based terminal cost and a decoupled positively invariant terminal set. The framework guarantees recursive feasibility for every bounded disturbance realization. Exploiting the potential-game structure induced by tracking costs, we further establish asymptotic convergence of each agent's nominal state to a steady-state variational generalized Nash equilibrium (vGNE), and show that each agent's actual state converges to a neighborhood of the vGNE determined by the minimal robust positively invariant set.
\end{abstract}


\section{Introduction}\label{sec:intro}

Modern engineered systems such as smart power grids, autonomous vehicles, and industrial robots increasingly rely on the interplay of competition and coordination among many agents operating in a shared uncertain environment. As the number of agents grows, relying on a single centralized controller becomes intractable. Instead, each agent needs to optimize its own objective while satisfying both private constraints and constraints it shares with others. For an individual agent, deploying a receding horizon or model predictive control (MPC) scheme is a natural choice because it explicitly handles constraints and provides rigorous guarantees on closed-loop behavior. Therefore, the study of game-theoretic interaction among agents that individually run receding horizon controllers is a problem of considerable practical relevance.

Consequently, the receding horizon games (RHG) framework has attracted growing interest in recent years. References \cite{wang2021,spica2020,liniger2019} explored game-theoretic planning for autonomous racing and self-driving cars, while Mignoni \emph{et al.}~\cite{mignoni2023} applied related ideas to energy management. The first rigorous treatments of recursive feasibility and stability in RHGs appeared in \cite{hall2022,hall2024}. The initial work~\cite{hall2022} established recursive feasibility and stability through a connection to potential games but relied on terminal equality constraints. The follow-up work~\cite{hall2024} proved stability under input coupling and input constraints for general economic costs, and did not consider constraints on the states or disturbance affecting the dynamics. Benenati and Grammatico~\cite{benenati2025} subsequently studied linear-quadratic dynamic games as receding-horizon variational inequalities.

The above works consider deterministic linear dynamical systems that are not affected by any disturbance or uncertainty. In the single-agent setting, robust MPC under additive disturbances have been studied extensively \cite{mayne2000constrained,borrelli2017predictive,rmpc_survey_bemporad2007robust}. Among various robust MPC techniques, tube-based approaches~\cite{mayne2005,kouvaritakis2015,rakovic2023} provide an elegant solution by confining the uncertain states to a tube around a nominal trajectory, and establish recursive feasibility and stability guarantees by tightening the constraints. 

Robust MPC for multiple agents has been studied largely in the \emph{cooperative}
distributed setting, where agents jointly minimize a system-wide objective rather
than their own costs. The tube-based treatment of coupled constraints under persistent disturbances in distributed MPC originated with \cite{richards2007,trodden2010}. Building on this foundation, \cite{kuwata2011} proposed a cooperative receding-horizon MILP scheme with constraint tightening, Conte \emph{et al.}~\cite{conte2013} developed tube-based
robust distributed MPC via distributed constraint tightening, and Stewart \emph{et
al.}~\cite{stewart2010} established cooperative distributed MPC that converges to the
centralized optimum. Trodden and Richards~\cite{trodden2013} promoted cooperation by
letting each agent design hypothetical plans for its neighbors while allowing only
one agent to optimize per step. A more recent work \cite{wang2022} combined tube-based tightening with online-adaptive local terminal sets, and \cite{wiltz2025} decoupled the coupled state constraints through consistency constraints that confine each agent to a communicated reference tube. Similarly, \cite{kohler2022} proposed to coordinate agents by communicated reference trajectories and a coupling cost. In all of these works, coupling is resolved by exchanging predicted or reference trajectories or dual variables, and feasibility relies either on a shared objective or on knowledge of neighbors' models. Furthermore, convergence is to a prescribed setpoint or a cooperative optimum, not to a game-theoretic equilibrium.

Despite this large body of work on robust MPC for single and cooperative multi-agent
systems, receding-horizon games (RHGs) under uncertainty have not been rigorously
investigated. The problem is challenging because, unlike in the cooperative setting,
self-interested agents in a game are unwilling to disclose their dynamics, costs, or
predicted trajectories. In addition, the target is not a prescribed setpoint but the
equilibrium of the coupled game itself, which must be reached without violating the uncertain constraints, involving the states and inputs, that are shared among multiple agents. In this paper, we take a step towards closing this gap. Our contributions are threefold.
\begin{enumerate}
  \item We formulate a RHG in which each agent has linear dynamics affected by
        bounded additive disturbance and must robustly satisfy private and shared
        constraints. We present a tube-based constraint-tightening
        scheme for the coupled constraints in which each agent broadcasts only a
        scalar tightening term and discloses neither its dynamics nor its cost to others,
        unlike the cooperative schemes of~\cite{trodden2013, wang2022, wiltz2025},
        where coupling is resolved by exchanging predicted or reference trajectories.

  \item We construct three terminal ingredients that together guarantee recursive
        feasibility for every bounded disturbance realization: a discrete algebraic
        Riccati equation (DARE)-based terminal cost, a decoupled positively invariant
        terminal set, and a resource-allocation scheme that splits the coupled
        terminal constraint into per-agent shares so that terminal invariance is
        certified. In contrast, the terminal sets in~\cite{wang2022, wiltz2025} are local and do not contend with a coupled terminal constraint.

  \item We provide a rigorous stability analysis based on the potential-game structure
        of the generalized Nash equilibrium problem (GNEP), where the game potential
        acts as a joint Lyapunov function. This establishes convergence of each
        agent's nominal state to a steady-state variational generalized Nash
        equilibrium (vGNE) and of the actual state to a minimal robust positively
        invariant (RPI) neighborhood of the vGNE. To the best of our knowledge, this
        is the first robust RHG with such guarantees.
\end{enumerate}

\subsubsection*{Notation}
We write $\col(x_i)_{i=1}^N$ to denote the concatenation of
vectors $x_i$, and $\mathbb{S}^n_{\succ 0}$ to denote the cone of symmetric
positive definite matrices. For a symmetric matrix $Q$, we define $\|x\|_Q^2 := x^\top Q x$. The symbols $\oplus$ and $\ominus$ denote the Minkowski sum and Pontryagin difference of sets, and $\inte(\cX)$ denotes the interior of~$\cX$. In MPC formulation, $y_i^{k|t}$ denotes the quantity $y$ associated with agent~$i$, $k$~steps ahead of the current time~$t$. We drop indices for better readability when the meaning is clear from the context.

\section{Problem Formulation}\label{sec:problem}
Let $\cA := \{1, \ldots, N\}$ denote a collection of $N$ agents. Agent~$i$ controls a discrete-time linear system whose state evolves as
\begin{equation}\label{eq:dyn}
  x_i^{t+1} = A_i\,x_i^t + B_i\,u_i^t + w_i^t,
  \qquad w_i^t \in \cW_i,
\end{equation}
where $x_i^t \in \R^{n_x}$, $u_i^t \in \R^{n_u}$, and
$w_i^t \in \cW_i \subset \R^{n_x}$ denote the state, control input and additive disturbance affecting the system at time $t$, with $\cW_i$ being a bounded set. 

The agents are allowed to exchange information over an undirected communication graph. Formally, let $\cG := (\cA, \mathcal{E})$ be a graph where each agent is associated with a node, and let $\cN_i \subseteq \cA$ denote the set of neighbors of agent~$i$ including $i$.

Let $x_i^{0:H|t} := (x_i^{0|t},\dots,x_i^{H|t})$ and
$u_i^{0:H-1|t} := (u_i^{0|t},\dots,u_i^{H-1|t})$ denote agent $i$'s predicted
state and input sequences over the horizon. At each time $t$, agent $i$ aims to
solve the following finite-horizon optimal control problem:
\begin{subequations}\label{eq:uncertain-game}
\begin{align}
  \min_{x_i^{0:H|t},\, u_i^{0:H-1|t}} \;\; & \sum_{k=0}^{H-1}
    \|x_i^{k|t} - s_i^*\|_{Q_i}^2 + \|u_i^{k|t} - v_i^*\|_{R_i}^2
    \label{eq:ug-cost}\\
  \text{s.t.}\;\; & x_i^{k+1|t} = A_i x_i^{k|t} + B_i u_i^{k|t} + w_i^{k|t},
    \label{eq:ug-dyn}\\
  & x_i^{k|t} \in \cX_i,\;\; u_i^{k|t} \in \cU_i, \label{eq:ug-priv}\\
  & \textstyle\sum_{j \in \cN_i} F_j x_j^{k|t} + G_j u_j^{k|t} \leq b_i,
    \label{eq:ug-shared}
\end{align}
\end{subequations}
where $x_i^{0|t} = x_i^t$, the dynamics~\eqref{eq:ug-dyn}, the private
constraints~\eqref{eq:ug-priv}, and the shared constraints~\eqref{eq:ug-shared}
hold for $k = 0,\dots,H-1$. The constraints~\eqref{eq:ug-priv} are local to
agent $i$, and the constraints in~\eqref{eq:ug-shared} are shared constraints
defined over its local neighborhood, with $F_j \in \R^{n_c \times n_x}$,
$G_j \in \R^{n_c \times n_u}$, and $b_i \in \R^{n_c}$.
Since the shared constraint~\eqref{eq:ug-shared} couples the feasibility set of agent $i$ with the feasibility set of its neighbors', each agent solving \eqref{eq:uncertain-game} forms a receding-horizon generalized Nash equilibrium problem (GNEP) under uncertainty.

Due to the presence of disturbance terms, the states are uncertain, and hence, we require the constraints to hold for every disturbance realization $w_i^{k|t} \in \cW_i$. In \eqref{eq:ug-cost}, we consider a tracking cost function where the tuple $(s_i^*, v_i^*)$ is a variational generalized Nash equilibrium (vGNE) of a steady-state game played between the agents, discussed subsequently.

\begin{problem} \label{problem_statement}
Our goal is to enable the agents to compute an open-loop vGNE of the RHG in a distributed manner such that 
\begin{itemize}
    \item the private and shared constraints are satisfied for all agents in a robust manner, i.e., for all $w_i^{k|t} \in \cW_i$,
    \item the closed-loop system converges to a neighborhood of $(s_i^*, v_i^*)$ for all agents, and
    \item agents do not share information about their dynamics and constraints with others, i.e.,  $(A_i, B_i, \cX_i, \cU_i, \cW_i, F_i, G_i, Q_i, R_i)$ are private to agent~$i$, whereas $b_i$ is common knowledge within the neighborhood $\cN_i$. 
\end{itemize}    
\end{problem}

This problem is more challenging as compared to distributed cooperative robust MPC problems because the coupled constraints must be enforced even though each agent's data $(A_i, B_i, \cX_i, \cU_i, \cW_i, F_i, G_i, Q_i, R_i)$ are private. Our analysis is carried out under the following assumptions.

\begin{assumption}\label{ass:main}
The graph $\cG$ is connected, and the following conditions hold for each $i \in \cA$:
\begin{enumerate}[label=(\roman*)]
  \item $\cW_i \subset \R^{n_x}$ is compact and convex with
    $0 \in \inte(\cW_i)$,
  \item  $\cX_i$ and $\cU_i$ are nonempty, compact, and convex, and
  \item the pair $(A_i, B_i)$ is stabilizable, and the state and input weights satisfy
      $Q_i \succ 0$ and $R_i \succ 0$. 
\end{enumerate}
\end{assumption}

We start by introducing the steady-state GNEP problem that leads to the cost function stated in \eqref{eq:ug-cost}. 

\section{Steady-state vGNE}\label{sec:vGNE}

We assume that the agents aim to arrive at the generalized Nash equilibrium of a steady-state game where an agent $i \in \cA$ solves the following constrained optimization problem:
\begin{subequations}\label{eq:ss}
\begin{align}
\min_{s_i,\, v_i} &\; \tilde{l}_i^p(s_i, v_i)
\quad\text{s.t.}\;\;
s_i = A_i s_i + B_i v_i, \label{eq:ss_dyn}\\
& s_i \in \cS_i,\;\; v_i \in \cV_i,\;\;
\sum_{j \in \cN_i} F_j s_j + G_j v_j \leq \bar{b}_i,
\label{eq:ss_constr}
\end{align}
\end{subequations}
where $s_i \in \R^{n_x}, v_i \in \R^{n_u}$ and the quantities $\cS_i, \cV_i$ and $\bar{b}_i$ are defined in the following section by adopting a tube-based constraint tightening approach. 
The cost function $\tilde{l}_i^p$ is assumed to satisfy the following conditions. 

\begin{assumption} \label{assumption_str_convex}
For each $i \in \cA$, the function $\tilde{l}_i^p: \R^{n_x} \times \R^{n_u} \to \R$ is continuously differentiable and strongly convex.   
\end{assumption}

Since the shared constraint~\eqref{eq:ss_constr} couples each agent's feasible set to its neighbors' decisions, \eqref{eq:ss} is a generalized Nash equilibrium problem. Let $z_i := (s_i, v_i)$, $z := \col(z_i)_{i\in\cA}$, and $\mathcal{Z} := \{ z \mid s_i = A_i s_i + B_i v_i, s_i \in \cS_i,\;\; v_i \in \cV_i, \sum_{j \in \cN_i} F_j s_j + G_j v_j \leq \bar{b}_i, \, \text{for all} \, \, i \in \cA \}$ be the set of strategies of the players that satisfy all the constraints. We now formally define the variational generalized Nash equilibrium (vGNE) of this game, and establish its existence and uniqueness.

\begin{definition}\label{def:ss_vgne}
A vector $z^* \in \cZ$ is a \emph{variational generalized Nash equilibrium}
(vGNE) if it solves the variational inequality $\mathrm{VI}(\cZ, F)$, i.e., $F(z^*)^{\top} (z - z^*) \geq 0$ for all $z \in \cZ$, where $F(z) := \col\big(\nabla_{z_i}\tilde{l}_i^p\big)_{i\in\cA}$ is the pseudo-gradient.
\end{definition}

\begin{proposition}[Existence and Uniqueness]\label{prop:vgne}
Let $\cS_i$ and $\cV_i$ be nonempty, compact and convex for all $i \in \cA$, $\mathcal{Z}$ be nonempty; and let Assumption \ref{assumption_str_convex} hold. Then, there exists a unique vGNE $z^*$ of the game defined in \eqref{eq:ss}.
\end{proposition}

\begin{proof}
The feasible set $\cZ$ is the intersection of the compact convex sets $\cS_i \times \cV_i$, the affine steady-state manifolds, and the half-spaces of the tightened shared constraint; hence $\cZ$ is compact and convex, and it is nonempty by hypothesis. Each $\tilde{l}_i^p$ is continuously differentiable, strongly convex in $z_i$, and depends only on $z_i$; hence the pseudo-gradient $F(z) = \col(\nabla_{z_i}\tilde{l}_i^p)_{i\in\cA}$ is continuous and strongly monotone on $\cZ$ with modulus $\mu = \min_{i\in\cA}\mu_i > 0$. Since strong monotonicity is $\xi$-monotonicity with $\xi = 2$, the variational
inequality $\mathrm{VI}(\cZ, F)$ stated in Definition~\ref{def:ss_vgne} admits a unique solution $z^*$ by~\cite[Thm.~2.3.3(b)]{VI_book}.
\end{proof}

In this work, we refer to the vGNE of \eqref{eq:ss} as the {\it steady-state} vGNE of the original problem \eqref{eq:uncertain-game} because at $z^*$, the dynamics of each agent is at a steady-state in the absence of disturbance. Each agent $i$ computes its component $(s_i^*, v_i^*)$ of the vGNE offline using a distributed algorithm, such as the distributed operator-splitting algorithm from \cite{yi2019}. In the following section, we present the tube-based constraint tightening scheme and compute the terminal ingredients for the uncertain RHG \eqref{eq:uncertain-game} that tracks the steady-state vGNE. 

\section{Tube-Based Reformulation of Robust RHG}\label{sec:tube}
We now present the tube-based distributed reformulation of \eqref{eq:uncertain-game} in this section.
To handle the additive disturbance robustly, we follow the rigid-tube
approach of~\cite{rakovic2023,kouvaritakis2015} and decompose the
predicted state and input into nominal and error components given by
$x_i^{k|t} = s_i^{k|t} + e_i^{k|t}$ and
$u_i^{k|t} = v_i^{k|t} + K_i^p\,e_i^{k|t}$, where $K_i^p$ is a
prestabilizing feedback gain. The nominal part captures the
deterministic prediction and the error part absorbs the uncertainty. Under this decomposition, the predicted state and input of agent~$i$
evolve as
\begin{align}
  s_i^{k+1|t} &= A_i\,s_i^{k|t} + B_i\,v_i^{k|t},
    \label{eq:nom}\\
  e_i^{k+1|t} &= \Phi_i^p\,e_i^{k|t} + w_i^{k|t},
    \label{eq:err}
\end{align}
where $K_i^p$ is chosen such that the closed-loop matrix $\Phi_i^p := A_i + B_i K_i^p$ is Schur stable; this is possible because $(A_i, B_i)$ is assumed to be stabilizable, and ensures that the error dynamics remains bounded for every admissible disturbance realization. The nominal recursion~\eqref{eq:nom} is deterministic and serves as the decision variable in the receding horizon GNEP. 

Under this choice and the compactness
of $\cW_i$, the error dynamics~\eqref{eq:err} have a unique \emph{minimal
robust positively invariant} (mRPI) set given by
\begin{equation}\label{eq:mrpi}
  \cZ_i^\infty := \bigoplus_{n=0}^{\infty} (\Phi_i^p)^n \cW_i,
\end{equation}
which is compact and satisfies $\Phi_i^p e + w \in \cZ_i^\infty$ for every
$e \in \cZ_i^\infty$ and $w \in \cW_i$~\cite{rakovic2005}. 
Each agent computes $\cZ_i^\infty$, or a robust outer approximation of it, locally
via the algorithm of~\cite{rakovic2005}. Consequently, whenever $e_i^{0|t} \in \cZ_i^\infty$, the error recursion~\eqref{eq:err} remains inside $\cZ_i^\infty$ for every disturbance realization, and the nominal
recursion~\eqref{eq:nom} evolves deterministically.

We exploit this invariance to tighten the private and shared constraints as follows.
\begin{enumerate}
  \item \emph{Private constraint tightening.}
    Since the error stays inside $\cZ_i^\infty$ whenever
    $e_i^{0|t} \in \cZ_i^\infty$, we define
    \begin{equation}\label{eq:tight_priv}
      \cS_i := \cX_i \ominus \cZ_i^\infty, \qquad
      \cV_i := \cU_i \ominus K_i^p \cZ_i^\infty.
    \end{equation}
    Any nominal trajectory satisfying these tightened sets therefore
    guarantees that the actual state and input respect the original
    constraints~\eqref{eq:ug-priv}, regardless of the disturbance.

\item \emph{Shared constraint tightening.}
The shared constraint requires more care, because every neighboring agent's
error contributes to the worst-case left-hand side of~\eqref{eq:ug-shared}.
As the constraint has $n_c$ rows, we tighten it row by row. Let $c_{j,\ell}^\top$
denote the $\ell$-th row of $F_j + G_j K_j^p$; the worst-case contribution of
agent~$j$ to row $\ell$ is the scalar
\begin{equation}\label{eq:dj}
  [d_j]_\ell := \max_{e_j \in \cZ_j^\infty}\; c_{j,\ell}^\top e_j,
  \qquad \ell = 1, \dots, n_c,
\end{equation}
that is, the support function of $\cZ_j^\infty$ in the direction $c_{j,\ell}$.
Stacking these scalars gives the vector $d_j \in \R^{n_c}$, which agent~$j$
computes from its own data alone. Subtracting the aggregate worst case yields
the tightened shared constraint
\begin{equation}\label{eq:tight_shared}
  \sum_{j \in \cN_i} F_j s_j^{k|t} + G_j v_j^{k|t} \;\leq\; \bar{b}_i,
  \, \,
  \bar{b}_i := b_i - \sum_{j \in \cN_i} d_j,
\end{equation}
where $\bar{b}_i \in \R^{n_c}$ and the inequality holds component-wise.
\end{enumerate}

Note that these tightened quantities $\cS_i$, $\cV_i$, and $\bar{b}_i$ were used in the steady-state game \eqref{eq:ss_constr} of Section \ref{sec:vGNE}. 

The shared tightening~\eqref{eq:tight_shared} requires each agent to know the bound $\bar{b}_i$, which aggregates the neighbors' worst-case error contributions. The following remark shows that $\bar{b}_i$ can be computed without disclosing the private data of any agent, as required in Problem~\ref{problem_statement}.

\begin{remark}[Privacy-Preserving Protocol]\label{rem:priv}
Computing $\bar{b}_i$ requires the values of $d_j$ for every neighbor $j \in \cN_i$. Each agent~$j$ computes its own $d_j$ locally and broadcasts only this scalar (or vector) to its neighbors. The neighbors, therefore, do not need to be aware of $(F_j, G_j, K_j^p, \cZ_j^\infty)$, which preserves the privacy of each agent's local data sought in Problem~\ref{problem_statement}.
\end{remark}

The tube reformulation and constraint tightening developed above set up the deterministic nominal problem. To obtain closed-loop guarantees, such as recursive feasibility and convergence, we additionally require three terminal ingredients, that is, a terminal control law together with a terminal cost and a terminal invariant set. We construct these next. 

\subsection{Terminal Ingredients} \label{sec:terminal} 

We obtain the terminal control law and terminal cost from the
unconstrained infinite-horizon LQR with state and input weights $Q_i \succ 0$, $R_i \succ 0$ following \cite{kouvaritakis2015}. For each $i \in \mathcal{A}$,
let $P_i \succ 0$ be the unique stabilizing solution of the discrete
algebraic Riccati equation (DARE),
\begin{equation}
  P_i = Q_i + A_i^\top P_i A_i
        - A_i^\top P_i B_i \bigl(R_i + B_i^\top P_i B_i\bigr)^{-1}
          B_i^\top P_i A_i ,
  \label{eq:dare}
\end{equation}
and let
\begin{equation}
  K^f_i := -\bigl(R_i + B_i^\top P_i B_i\bigr)^{-1} B_i^\top P_i A_i
  \label{eq:terminal_gain}
\end{equation}
be the corresponding optimal feedback gain. Under
Assumption~\ref{ass:main}(iii), i.e. stabilizability of
$(A_i, B_i)$, together with $Q_i \succ 0$ and $R_i \succ 0$, a
stabilizing solution $P_i \succ 0$ of \eqref{eq:dare} exists, is unique,
and renders the closed-loop matrix $\Phi^f_i := A_i + B_i K^f_i$ Schur
stable. We define the terminal control law as
\begin{equation}
  \kappa_i(s_i) := v^*_i + K^f_i\,(s_i - s^*_i),
  \label{eq:kappa}
\end{equation}
and the terminal cost as
\begin{equation}
  l^f_i(s_i) := \lVert s_i - s^*_i \rVert^2_{P_i}.
  \label{eq:terminal_cost}
\end{equation}
At the optimal gain \eqref{eq:terminal_gain}, the DARE \eqref{eq:dare}
is equivalent to the closed-loop Lyapunov identity
\begin{equation}
  P_i = Q_i + (K^f_i)^\top R_i K^f_i
        + (\Phi^f_i)^\top P_i\, \Phi^f_i ,
  \label{eq:closed_loop_dare}
\end{equation}
which we subsequently leverage. 

Having fixed the terminal control law and cost, it remains to construct a
terminal set that is positively invariant under $\kappa_i$ while respecting the
shared constraint. The difficulty is that the shared constraint couples the
agents, so verifying invariance would ordinarily require coordination among
neighbors. To avoid this, we split the tightened shared bound $\bar b_i$
from~\eqref{eq:tight_shared} into per-agent shares. Since an agent may belong to
several neighborhoods, we let $\alpha_j^{(i)}$ denote the share that the
constraint of agent $i$ assigns to agent $j$, with
\begin{equation}\label{eq:alpha}
  F_j s_j + G_j v_j \le \alpha_j^{(i)}, \;\; j \in \cN_i,
  \quad \sum_{j \in \cN_i}\alpha_j^{(i)} = \bar b_i.
\end{equation}
Since agent $j$ maintains a single terminal set, it enforces the tightest
share demanded by any constraint it participates in. Let
$\cM_j := \{\, i \in \cA : j \in \cN_i \,\}$ denote the set of neighborhoods
containing agent $j$. The enforced share is then given by
\begin{equation}\label{eq:minshare}
  \alpha_j := \min_{i \in \cM_j}\, \alpha_j^{(i)},
\end{equation}
so that $F_j s_j + G_j v_j \le \alpha_j \le \alpha_j^{(i)}$ for every $i \in \cM_j$.
Each agent then enforces its own share $\alpha_j$ independently.

For this decoupling to be compatible with the steady-state operating point, an explicit choice of the per-neighborhood shares is given as follows.
\begin{equation}
  \alpha_j^{(i)} = F_j s_j^* + G_j v_j^* + \frac{1}{|\mathcal{N}_i|}
  \Big[\bar b_i - \sum_{k\in\mathcal{N}_i}\big(F_k s_k^* + G_k v_k^*\big)\Big],
  \, \, j\in\mathcal{N}_i, \label{eq:allocation}
\end{equation}
which satisfy $\sum_{j\in\cN_i}\alpha_j^{(i)} = \bar b_i$ by construction, so the
shares partition the tightened bound $\bar b_i$ among the neighbors of agent~$i$.
These shares are computed offline, alongside the steady-state vGNE. To evaluate
its own share $\alpha_j$, agent~$j$ needs, for each neighborhood $i\in\cM_j$ it
belongs to, only the scalar contributions $F_k s_k^* + G_k v_k^*$ of the members
$k\in\cN_i$, which are broadcast within that neighborhood. By leveraging this information, agent~$j$ computes $\alpha_j^{(i)}$ via \eqref{eq:allocation} and then $\alpha_j$
via \eqref{eq:minshare}, without any further coordination.

\begin{lemma}\label{lem:vgne-share}
For each $j\in\cA$, the steady-state vGNE satisfies $s_j^*\in\cS_j$,
$v_j^*\in\cV_j$, and $F_j s_j^* + G_j v_j^* \le \alpha_j$, with $\alpha_j$
given by \eqref{eq:minshare}.
\end{lemma}
\begin{proof}
By Proposition~\ref{prop:vgne}, the steady-state vGNE exists and satisfies all
constraints of \eqref{eq:ss}. In particular, $s_j^*\in\cS_j$ and $v_j^*\in\cV_j$
for every $j\in\cA$. Fix any $i\in\cA$. The shared constraint \eqref{eq:ss_constr}
gives $\bar b_i - \sum_{k\in\cN_i}\big(F_k s_k^* + G_k v_k^*\big) \ge 0$, so the
bracketed term in the allocation \eqref{eq:allocation} is nonnegative, yielding
$F_j s_j^* + G_j v_j^* \le \alpha_j^{(i)}$ for each $j\in\cN_i$. Since this holds
for every neighborhood $i\in\cM_j$ containing $j$, it holds in particular for the
one attaining the minimum in \eqref{eq:minshare}; therefore
$F_j s_j^* + G_j v_j^* \le \alpha_j$.
\end{proof}

Given the admissible shares $\{\alpha_j\}_{j\in\mathcal{N}_i}$, we now define the
terminal set for each agent as follows.
\begin{definition}[Terminal Set]\label{def:Sf}
The terminal set $\cS_i^f$ for agent~$i$ is the largest set
$\cS_i^f \subseteq \cS_i$ such that, for every $s_i \in \cS_i^f$,
\begin{equation}\label{eq:Sf}
  \kappa_i(s_i) \in \cV_i,\quad
  F_i s_i + G_i \kappa_i(s_i) \le \alpha_i,\quad
  A_i s_i + B_i \kappa_i(s_i) \in \cS_i^f.
\end{equation}
\end{definition}

The set $\cS_i^f$ is the maximal positively invariant set for the closed-loop
nominal dynamics under $\kappa_i$ subject to the local and decoupled
constraints. It is nonempty, since the steady-state vGNE satisfies all three
conditions in~\eqref{eq:Sf}, that is, $\kappa_i(s_i^*) = v_i^* \in \cV_i$,
$F_i s_i^* + G_i v_i^* \le \alpha_i$ by Lemma~\ref{lem:vgne-share}, and $A_i s_i^* + B_i v_i^* = s_i^*$ by the steady-state identity~\eqref{eq:ss_dyn};
hence $s_i^* \in \cS_i^f$.

\begin{remark}[Computation of $\cS_i^f$] \label{rem:Sif_comp} 
When $\cX_i$ and $\cU_i$ are polytopic and $\cZ_i^\infty$ is replaced by a
polytopic outer approximation~\cite{rakovic2005}, the tightened sets
$\cS_i, \cV_i$ are polytopes and $\cS_i^f$ is the maximal positively invariant
set of the affine recursion $s_i^{k+1} = \Phi_i^f s_i^k + B_i(v_i^* - K_i^f
s_i^*)$, computed locally via the Multi-Parametric Toolbox~\cite{mpt3}. For
ellipsoidal or other convex constraints, invariant sets can instead be computed
using LMI-based methods~\cite{boyd1994linear,blanchini1999set}.
\end{remark}

\subsection{Robust Receding Horizon GNEP}\label{sec:gnep}

With the tightened constraints and terminal ingredients in place, we can now state the receding-horizon game that each agent solves online. Since the goal of the agents is to converge to a neighborhood of the steady-state vGNE by repeatedly solving a receding horizon problem, we define a tracking cost 
\begin{equation} \label{eq:rh_tracking_cost}
l_i^p(s_i, v_i)
:= \|s_i - {s}^*_i\|_{Q_i}^2 + \|v_i - {v}^*_i\|_{R_i}^2,   
\end{equation}
where $Q_i \succ 0$ and $R_i \succ 0$.

Let $s_i^{0:H|t}:= (s_i^{0|t},\dots,s_i^{H|t})$ and
$v_i^{0:H-1|t} := (v_i^{0|t},\dots,v_i^{H-1|t})$ denote agent $i$'s predicted
nominal state and input sequences, and let $z_i^t := (s_i^{0:H|t}, v_i^{0:H-1|t})$.
At each time~$t$, agent $i$ solves the following finite-horizon optimization
problem in parallel with its peers.
\begin{subequations}\label{eq:gnep}
\begin{align}
  \min_{z_i^t} &\;\;
    J_i (z_i^t) := \sum_{k=0}^{H-1} l_i^p(s_i^{k|t}, v_i^{k|t})
    + l_i^f(s_i^{H|t})
    \label{eq:gnep_cost}\\
  \text{s.t.}\;\;
   & s_i^{k+1|t} = A_i s_i^{k|t} + B_i v_i^{k|t},
    \label{eq:gnep_dyn}\\
   & s_i^{k|t} \in \cS_i,\; v_i^{k|t} \in \cV_i,
    \label{eq:gnep_priv}\\
   & \textstyle\sum_{j \in \cN_i}  F_j s_j^{k|t} + G_j v_j^{k|t} \leq \bar{b}_i,
    \label{eq:gnep_shared} \\
   & s_i^{H|t} \in \cS_i^f,
    \label{eq:gnep_term}\\
   & x_i^t -  s_i^{0|t} \in \cZ_i^\infty,
    \label{eq:gnep_tube}
\end{align}
\end{subequations}
where the dynamics~\eqref{eq:gnep_dyn}, the tightened private
sets~\eqref{eq:gnep_priv}, and the tightened shared bound~\eqref{eq:gnep_shared}
hold for $k = 0,\dots,H-1$ and the stage cost
$l_i^p$ and terminal cost $l_i^f$ are given by~\eqref{eq:rh_tracking_cost}
and~\eqref{eq:terminal_cost}, respectively.

The following remark presents an important feature of the stage cost \eqref{eq:rh_tracking_cost} involved in the cost function \eqref{eq:gnep_cost}.

\begin{remark}\label{rem:tracking_cost}
The stage cost $l_i^p$ in~\eqref{eq:rh_tracking_cost} is the standard
set-point regulation objective of MPC; it penalizes deviation from the
desired operating point $(s_i^*, v_i^*)$, with the weights $Q_i$ and
$R_i$ acting as free design parameters. This particular form is essential
for our stability analysis presented in Section~\ref{sec:stability}.
\end{remark}

Following the rigid-tube approach~\cite{rakovic2023}, we treat the initial nominal state $s_i^{0|t}$ as a decision variable rather than fixing it to the measured state. The tube constraint~\eqref{eq:gnep_tube} only requires $x_i^t$ to lie within $\cZ_i^\infty$ of the chosen $s_i^{0|t}$, which gives the optimizer additional freedom and enlarges the region of attraction. We denote the joint open-loop vGNE of~\eqref{eq:gnep} by $(s_i^{*,0:H|t}, v_i^{*,0:H-1|t})_{i \in \cA}$, which the agents compute in a distributed fashion via the operator-splitting methods of~\cite{yi2019,belgioioso2022}. Having solved the game, each agent applies the input,
\begin{equation}\label{eq:control}
  u_i^t = v_i^{*,0|t} + K_i^p(x_i^t - s_i^{*,0|t}),
\end{equation}
which combines the optimal nominal input $v_i^{*,0|t}$ with the pre-stabilizing correction $K_i^p(x_i^t - s_i^{*,0|t})$ that drives the actual state back towards the nominal one.
The complete scheme separates naturally into an offline design phase, carried out once, and an online phase repeated at each time step. The summary of the complete scheme is presented in Algorithm~\ref{alg:main}.

\begin{algorithm}[t]
\caption{Robust Receding Horizon Game}\label{alg:main}
\begin{algorithmic}[1]
\STATE \textbf{Offline (each agent $i$):}
\STATE Choose $Q_i \succ 0$, $R_i \succ 0$ and select the prestabilizing gain $K_i^p$.
\STATE Compute $\cZ_i^\infty$ via~\cite{rakovic2005} and tighten $\cS_i$,
  $\cV_i$, and $\bar{b}_i$
  via~\eqref{eq:tight_priv}--\eqref{eq:tight_shared}.
\STATE Compute the steady-state vGNE $(s_i^*, v_i^*)$ of~\eqref{eq:ss}
  using~\cite{yi2019}.
\STATE Compute $P_i$ and $K_i^f$ from \eqref{eq:dare} and \eqref{eq:terminal_gain}, respectively, and then compute the terminal set $\cS_i^f$ following Remark~\ref{rem:Sif_comp}. 
\STATE \textbf{Online (at each time $t$):}
\STATE Measure $x_i^t$.
\STATE Solve the GNEP~\eqref{eq:gnep}
  in a distributed manner via~\cite{yi2019,belgioioso2022} to obtain
  $(s_i^{*,0:H|t}, v_i^{*,0:H-1|t})$.
\STATE Apply $u_i^t = v_i^{*,0|t} + K_i^p(x_i^t - s_i^{*,0|t})$.
\STATE Set $t \leftarrow t + 1$ and return to Step~7.
\end{algorithmic}
\end{algorithm}

The GNEP~\eqref{eq:gnep} has a strongly monotone pseudo-gradient over a convex feasible set on a connected graph (Assumptions~\ref{ass:main} and \ref{assumption_str_convex}), which are the conditions under which the operator-splitting methods of~\cite{yi2019,belgioioso2022} are guaranteed to converge to the vGNE.

\section{Recursive Feasibility and Closed-Loop Stability Analysis}\label{sec:rf}

\subsection{Recursive Feasibility}

We next show that the optimization problem~\eqref{eq:gnep} remains feasible at every future time step for every disturbance realization if it is feasible at $t=0$. Let $\cD(x^t)$ denote the joint feasibility set of~\eqref{eq:gnep} at state $x^t$, that is, the set of all nominal sequences $(s_i^{0:H|t}, v_i^{0:H-1|t})_{i \in \cA}$ satisfying constraints~\eqref{eq:gnep_dyn} to~\eqref{eq:gnep_tube} for all $i \in \cA$. We prove that if $\cD(x^t)$ is nonempty, then so is $\cD(x^{t+1})$ for any realization of the disturbance variables of the agents. To achieve this, we construct a feasible point of~\eqref{eq:gnep} at time $t+1$, called the shifted candidate solution, from the optimal solution at time $t$ by  dropping the first step of the optimal trajectory, shifting the rest forward by one step, and appending a final step generated by the terminal control law. We denote this candidate by $\hat{s}_i^{k|t+1}$ and $\hat{v}_i^{k|t+1}$, which is formally defined below. 

\begin{definition}[Shifted candidate solution]\label{def:cand}
Given the vGNE $(s_i^{*,0:H|t}, v_i^{*,0:H-1|t})_{i \in \cA}$
of~\eqref{eq:gnep} at time~$t$, we define for each agent $i$
\begin{align}
  \hat{s}_i^{k|t+1} &:=
  \begin{cases}
    s_i^{*,k+1|t}, & k = 0, \ldots, H\!-\!1, \\
    A_i s_i^{*,H|t} + B_i \kappa_i(s_i^{*,H|t}), & k = H,
  \end{cases}\label{eq:cand_s}\\
  \hat{v}_i^{k|t+1} &:=
  \begin{cases}
    v_i^{*,k+1|t}, & k = 0, \ldots, H\!-\!2, \\
    \kappa_i(s_i^{*,H|t}), & k = H\!-\!1.
  \end{cases}\label{eq:cand_v}
\end{align}
\end{definition}
We now show that this candidate is feasible at time $t+1$ for every disturbance realization, which establishes recursive feasibility.
\begin{theorem}[Recursive Feasibility]\label{thm:rf}
Suppose $\cD(x^0) \neq \emptyset$. Then, under the terminal ingredients chosen in Section \ref{sec:terminal}, the GNEP \eqref{eq:gnep} remains feasible for every $t \geq 0$ and every disturbance realization $w_i^t \in \cW_i$,
$i \in \cA$.
\end{theorem}
\begin{proof}
We show that $\cD(x^t) \neq \emptyset$ implies $\cD(x^{t+1}) \neq \emptyset$ by verifying that the shifted candidate~\eqref{eq:cand_v}--\eqref{eq:cand_s} is
jointly feasible for~\eqref{eq:gnep} at time $t+1$. We now verify each constraint.

\smallskip\noindent
\emph{Step~1 (Tube constraint~\eqref{eq:gnep_tube}).}
Let $e_i^{0|t} := x_i^t - s_i^{*,0|t}$, which satisfies $e_i^{0|t} \in
\cZ_i^\infty$ due to feasibility at time $t$. Since $\hat{s}_i^{0|t+1} =
s_i^{*,1|t}$ and $u_i^t$ is given by~\eqref{eq:control},
\begin{align}
  x_i^{t+1} - \hat{s}_i^{0|t+1}
  &= A_i x_i^t + B_i u_i^t + w_i^t
    - A_i s_i^{*,0|t} - B_i v_i^{*,0|t} \notag\\
  &= \Phi_i^p\,e_i^{0|t} + w_i^t \in \cZ_i^\infty,
    \label{eq:tube_pf}
\end{align}
where the inclusion holds because $\cZ_i^\infty$ is RPI and given by \eqref{eq:mrpi}. Thus, the tube constraint is satisfied.

\smallskip\noindent
\emph{Step~2 (Nominal dynamics~\eqref{eq:gnep_dyn}).}
The shifted state and input~\eqref{eq:cand_v}--\eqref{eq:cand_s} satisfy
the nominal dynamics by construction: for $k \leq H\!-\!1$ they inherit
the dynamics of the optimal trajectory at time $t$, and the appended step
$\hat{s}_i^{H|t+1} = A_i s_i^{*,H|t} + B_i \kappa_i(s_i^{*,H|t})$ matches
the dynamics under $\kappa_i$.

\smallskip\noindent
\emph{Step~3 (Private~\eqref{eq:gnep_priv} and terminal~\eqref{eq:gnep_term} constraints).}
For $k = 0, \ldots, H\!-\!1$, the shifted values $\hat{s}_i^{k|t+1} =
s_i^{*,k+1|t} \in \cS_i$ and $\hat{v}_i^{k|t+1} \in \cV_i$ were already
feasible at time $t$. For the appended step $k = H$, the positive
invariance of $\cS_i^f$ (Definition~\ref{def:Sf}) gives
$\hat{s}_i^{H|t+1} = A_i s_i^{*,H|t} + B_i \kappa_i(s_i^{*,H|t}) \in
\cS_i^f \subseteq \cS_i$, which satisfies both the private constraint at
$k = H$ and the terminal constraint.

\smallskip\noindent
\emph{Step~4 (Shared constraint~\eqref{eq:gnep_shared}).}
This is the only constraint that couples the agents, so it must be checked jointly. For $k = 0, \dots, H{-}2$, the shifted sequences of all agents satisfy
\begin{equation*}
  \sum_{j\in\cN_i} F_j \hat s_j^{k|t+1} + G_j \hat v_j^{k|t+1}
  = \sum_{j\in\cN_i} F_j s_j^{*,k+1|t} + G_j v_j^{*,k+1|t} \le \bar b_i,
\end{equation*}
since the step $k+1$ was feasible at time $t$. For the appended step
$k = H{-}1$, each agent's terminal state satisfies $s_j^{*,H|t}\in\cS_j^f$, so
Definition~\ref{def:Sf} and the enforced share~\eqref{eq:minshare} give
\begin{equation*}
  F_j s_j^{*,H|t} + G_j \kappa_j(s_j^{*,H|t}) \le \alpha_j \le \alpha_j^{(i)},
  \quad j \in \cN_i,
\end{equation*}
where the second inequality holds since $i \in \cM_j$ and $\alpha_j$ is, by
definition~\eqref{eq:minshare}, the minimum of the shares $\alpha_j^{(i)}$ over
all neighborhoods in $\cM_j$; hence $\alpha_j \le \alpha_j^{(i)}$.
Summing over the neighborhood and using
$\sum_{j\in\cN_i}\alpha_j^{(i)} = \bar b_i$ yields
\begin{equation*}
  \sum_{j\in\cN_i} F_j s_j^{*,H|t} + G_j \kappa_j(s_j^{*,H|t})
  \le \sum_{j\in\cN_i}\alpha_j
  \le \sum_{j\in\cN_i}\alpha_j^{(i)} = \bar b_i,
\end{equation*}
so the shared constraint holds at $k = H{-}1$ as well.
\smallskip

Having verified~\eqref{eq:gnep_dyn}--\eqref{eq:gnep_tube}, we conclude that the shifted
candidate is feasible at $t+1$, and therefore $\cD(x^{t+1}) \neq
\emptyset$.
\end{proof}

\subsection{Closed-Loop Stability}\label{sec:stability}

We now turn to the central stability question: do the agents'
nominal trajectories converge to the steady-state vGNE under the
receding horizon policy? We exploit the potential game structure to find a
joint Lyapunov function from the individual cost functions.

\begin{lemma}[Potential Game]\label{lem:potential}
The GNEP~\eqref{eq:gnep} is a generalized potential game with potential
function $V(z^t) := \sum_{i=1}^N J_i(z_i^t)$, where $z^t = \mathrm{col}(z_i^t)_{i\in\mathcal{A}}$
and $J_i$ is given from \eqref{eq:gnep_cost}. Furthermore, the vGNE is the unique
minimizer of $V$ over $\mathcal{D}(x^t)$.
\end{lemma}

\begin{proof}
The cost function $J_i$ from \eqref{eq:gnep_cost} depends only on agent $i$'s own predicted sequences $\{s_i^{k|t}\}_{k=0}^{H}$ and $\{v_i^{k|t}\}_{k=0}^{H-1}$, since the agents are coupled solely through the shared constraint \eqref{eq:gnep_shared}. Consequently, $\nabla_{z_i} J_i = \nabla_{z_i} V$ for every $i$, so the pseudo-gradient $F(z) = \mathrm{col}\big(\nabla_{z_i} J_i(z_i)\big)_{i \in \cA}$ coincides
with the gradient of $V$, i.e.,\ $F(z) = \nabla_z V$. This is precisely the defining property of a potential game with potential $V$. Furthermore, each $J_i$ is strongly convex in $(s_i, v_i)$
because $Q_i, R_i \succ 0$, so $V$ is strongly convex on the convex feasible set $\cD(x^t)$ and admits a unique minimizer. Since $F(z) =\nabla_z V$, this minimizer satisfies the same variational inequality that defines the vGNE (Definition \ref{def:ss_vgne}), and hence the two coincide.
\end{proof}
The following corollary formalizes that since the vGNE minimizes $V$, it is no worse than any other feasible point.

\begin{corollary}\label{cor:opt}
For any feasible point $\bar{z}^t\in \cD(x^t)$, we
have $V(z^{*,t}) \leq V(\bar{z}^t)$, where $z^{*,t}=(s^{*,0:H|t}_i,v^{*,:H-1|t}_i)_{i\in \cA}$ denotes the vGNE at time $t$.
\end{corollary}

We use the potential function $V$ introduced in Lemma~\ref{lem:potential} as a candidate Lyapunov function. Indeed, $V$ is positive definite with respect to the steady-state vGNE, with $V(z) = 0$ if and only if $z = z^*$ and $V(z) \ge 0$ otherwise, since $Q_i, R_i, P_i \succ 0$. 

To establish convergence, we need to evaluate $V$ along the closed-loop sequence of RHG vGNE solutions $z^{*,t}$ produced by Algorithm~\ref{alg:main} and show that the potential function is strictly decreasing along this trajectory. Before stating the main stability result, we present a one-step decrease property of the terminal cost~\eqref{eq:terminal_cost} under the terminal control law~\eqref{eq:kappa}. This property is useful to establish the main result.


\begin{lemma}[Terminal Cost Decrease]
\label{lem:terminal-decrease}
For every $s_i \in \R^{n_x}$,
\begin{equation}
l^p_i(s_i, \kappa_i(s_i)) + l^f_i\big(s^+_i\big) - l^f_i(s_i) = 0,
\label{eq:terminal-decrease}
\end{equation}
where $s^+_i=A_i s_i + B_i \kappa_i(s_i)$.
\end{lemma}

\begin{proof}
Let $\delta := s_i - s^*_i$. Then, following \eqref{eq:kappa}, we have
\begin{equation} \label{eq:kf_delta}
\kappa_i(s_i) - v^*_i = K^f_i \delta.    
\end{equation}
Next, following the steady-state identity $s^*_i = A_i s^*_i + B_i v^*_i$ from \eqref{eq:ss_dyn}, we obtain 
\begin{equation} \label{eq:phif_delta}
s^+_i-s^*_i= \Phi^f_i \delta.    
\end{equation}
Now, leveraging \eqref{eq:kf_delta} in the expression for $l^p_i(s_i, \kappa_i(s_i))$, we obtain
\begin{align}
l^p_i(s_i, \kappa_i(s_i))&=\|s_i-s^*_i\|^2_{Q_i} + \|\kappa_i(s_i)-v^*_i\|^2_{R_i}, \nonumber \\
&= \delta^{\top}\big[Q_i + (K^f_i)^{\top} R_i K^f_i\big]\delta. \label{eq:lp_i}
\end{align}
Similarly, leveraging \eqref{eq:phif_delta} in the expression for $l^f_i(s^+_i)$, we obtain
\begin{align} \label{eq:lfi_si_plus}
l^f_i(s^+_i)=\|s^+_i-s^*_i\|^2_{P_i}= \delta^{\top}(\Phi^f_i)^{\top} P_i \Phi^f_i \delta    
\end{align}
Furthermore, we have
\begin{align} \label{eq:lfi_si}
l^f_i(s_i)=\|s_i-s^*_i\|^2_{P_i}= \delta^{\top} P_i \delta   
\end{align}
Following \eqref{eq:lp_i}, \eqref{eq:lfi_si_plus} and \eqref{eq:lfi_si}, we get
$l^p_i(s_i, \kappa_i(s_i)) + l^f_i\big(s^+_i\big) - l^f_i(s_i)=\delta^{\top} \big[ Q_i + (K^f_i)^{\top} R_i K^f_i + (\Phi^f_i)^{\top} P_i \Phi^f_i - P_i\big] \delta$.
The sum vanishes following~\eqref{eq:closed_loop_dare}.
\end{proof}

We now state and prove the asymptotic convergence of the nominal
trajectories to the steady-state vGNE. 

\begin{theorem}[Robust Asymptotic Stability]\label{thm:RAS}
Suppose $\cD(x^0) \neq \emptyset$. Then, under Algorithm~\ref{alg:main}, for every realization $w_i^t \in \cW_i$, $i \in \cA$, the following hold.
\begin{enumerate}[label=(\roman*)]
  \item \emph{(Lyapunov decrease.)} For all $t \geq 0$,
    \begin{equation}\label{eq:lyap_decrease}
      V(z^{*,t+1}) - V(z^{*,t})
      \leq -\sum_{i=1}^N l_i^p(s_i^{*,0|t}, v_i^{*,0|t}).
    \end{equation}
  \item \emph{(Nominal convergence.)} For every $i \in \cA$,
    $\lim_{t \to \infty} \|s_i^{*,0|t} - s_i^*\| = 0$ and
    $\lim_{t \to \infty} \|v_i^{*,0|t} - v_i^*\| = 0$.
  \item \emph{(Actual-state convergence.)} For every $i \in \cA$, the actual state $x_i^t$ converges to the set $\{s_i^*\} \oplus \cZ_i^\infty$:
    \begin{equation}\label{eq:actual}
      \dist\bigl(x_i^t,\,\{s_i^*\}\oplus\cZ_i^\infty\bigr)
      \leq \|s_i^{*,0|t} - s_i^*\| \to 0
      \quad \text{as } t \to \infty.
    \end{equation}
\end{enumerate}
\end{theorem}

\begin{proof}
\emph{(i) Lyapunov decrease.}
We first evaluate agent $i$'s cost at the shifted candidate $\hat{z}_i^{t+1}$
from Definition~\ref{def:cand}, which is given as follows.
\begin{align}
  J_i(\hat{z}_i^{t+1}) &= \sum_{k=0}^{H-1} l_i^p(\hat{s}_i^{k|t+1}, \hat{v}_i^{k|t+1})
   + l_i^f(\hat{s}_i^{H|t+1}), \label{eq:Jhat_1}\\
   & = \sum_{k=0}^{H-2} l_i^p(s_i^{*,k+1|t}, v_i^{*,k+1|t})
    + l_i^p(s_i^{*,H|t}, \kappa_i(s_i^{*,H|t})) \nonumber \\
    & \quad + l_i^f(\hat{s}_i^{H|t+1}),\nonumber \\
   & = \sum_{k=1}^{H-1} l_i^p(s_i^{*,k|t}, v_i^{*,k|t})
    + l_i^p(s_i^{*,H|t}, \kappa_i(s_i^{*,H|t})) \nonumber \\
    & \quad + l_i^f(\hat{s}_i^{H|t+1}), \label{eq:Jhat_2} \\
   & = \sum_{k=1}^{H-1} l_i^p(s_i^{*,k|t}, v_i^{*,k|t})
    +  l_i^f({s}_i^{*,H|t}),  \label{eq:Jhat_3}
\end{align}
where \eqref{eq:Jhat_2} is obtained from \eqref{eq:Jhat_1} by first leveraging
\eqref{eq:cand_s} and \eqref{eq:cand_v} and then reindexing with $k=k+1$, and
\eqref{eq:Jhat_3} results from applying Lemma~\ref{lem:terminal-decrease} to
\eqref{eq:Jhat_2} with $s_i=s^{*,H|t}_i$.
Now, $J_i(z_i^{*,t})$ can be expressed in terms of
$J_i(\hat{z}_i^{t+1})$ from \eqref{eq:Jhat_3} as
\begin{align} \label{eq:Jstar_Jhat}
J_i(z_i^{*,t})= l_i^p(s_i^{*,0|t}, v_i^{*,0|t}) + J_i(\hat{z}_i^{t+1}). 
\end{align}
Rearranging and summing \eqref{eq:Jstar_Jhat} over $i$, we obtain
\begin{equation}\label{eq:Phihat}
  V(\hat{z}^{t+1})
  = V(z^{*,t}) - \sum_{i=1}^N l_i^p(s_i^{*,0|t}, v_i^{*,0|t}),
\end{equation}
where $\hat{z}^{t+1} = \mathrm{col}(\hat{z}_i^{t+1})_{i \in \cA}$ and
$V(\hat{z}^{t+1}) = \sum_i J_i(\hat{z}_i^{t+1})$. Since $\hat{z}^{t+1}$ is
feasible at $t+1$ (Theorem~\ref{thm:rf}), Corollary~\ref{cor:opt} gives
$V(z^{*,t+1}) \leq V(\hat{z}^{t+1})$. Combining this with \eqref{eq:Phihat}
yields \eqref{eq:lyap_decrease}.

\smallskip\noindent

\emph{(ii) Nominal convergence.}
By construction $V(z^{*,t}) \geq 0$, and~\eqref{eq:lyap_decrease} shows that it is
non-increasing. Summing~\eqref{eq:lyap_decrease} from $t=0$ to $T-1$ telescopes
the left-hand side, giving
\begin{equation*}
\sum_{t=0}^{T-1} \sum_i l_i^p(s_i^{*,0|t}, v_i^{*,0|t})
\;\leq\; V(z^{*,0}) - V(z^{*,T}) \;\leq\; V(z^{*,0}),
\end{equation*}
where the last step uses $V(z^{*,T}) \geq 0$. Letting $T \to \infty$ and
substituting $l_i^p$ from~\eqref{eq:rh_tracking_cost} yields
\begin{equation*}
\sum_{t=0}^{\infty} \sum_i \Big( \|s_i^{*,0|t} - s_i^*\|_{Q_i}^2
  + \|v_i^{*,0|t} - v_i^*\|_{R_i}^2 \Big)
\leq V(z^{*,0}) < \infty,
\end{equation*}
where $V(z^{*,0}) < \infty$ follows from feasibility at $t=0$. Hence, the series converges, i.e.,
\begin{equation*}
\|s_i^{*,0|t} - s_i^*\|_{Q_i}^2 + \|v_i^{*,0|t} - v_i^*\|_{R_i}^2 \to 0
\quad \text{as } t \to \infty.
\end{equation*}
Since each term is nonnegative, both summands vanish separately, i.e.,
$\|s_i^{*,0|t} - s_i^*\|_{Q_i}^2 \to 0$ and $\|v_i^{*,0|t} - v_i^*\|_{R_i}^2 \to 0$.
Using $x^\top Q_i x \geq \lambda_{\min}(Q_i)\|x\|^2$ with $Q_i \succ 0$, we obtain
\begin{equation*}
\|s_i^{*,0|t} - s_i^*\|^2 \leq \tfrac{1}{\lambda_{\min}(Q_i)}
\|s_i^{*,0|t} - s_i^*\|_{Q_i}^2 \to 0,
\end{equation*}
where $\lambda_{\min}(Q_i)$ denotes the smallest eigenvalue of $Q_i$. Thus,
$s_i^{*,0|t} \to s_i^*$ as $t \to \infty$. The same argument with $R_i \succ 0$
gives $v_i^{*,0|t} \to v_i^*$.

\smallskip\noindent

\emph{(iii) Actual-state convergence.}
By the tube constraint~\eqref{eq:gnep_tube}, at every time $t$ the actual state
satisfies $x_i^t \in \{s_i^{*,0|t}\} \oplus \cZ_i^\infty$, so the error
$e_i^t := x_i^t - s_i^{*,0|t}$ lies in $\cZ_i^\infty$. Consequently, the point
$s_i^* + e_i^t$ belongs to the target set $\{s_i^*\} \oplus \cZ_i^\infty$, and
bounding the distance from $x_i^t$ to the target set by the distance to this
particular point gives
\begin{equation*}
  \dist\big(x_i^t,\, \{s_i^*\}\oplus\cZ_i^\infty\big)
  \;\leq\; \big\|x_i^t - (s_i^* + e_i^t)\big\|
  = \big\|s_i^{*,0|t} - s_i^*\big\|,
\end{equation*}
where the error terms cancel since $x_i^t = s_i^{*,0|t} + e_i^t$. By part~(ii),
the right-hand side tends to zero as $t \to \infty$, so $x_i^t$ converges to
$\{s_i^*\} \oplus \cZ_i^\infty$, establishing~\eqref{eq:actual}.
\end{proof}
\begin{remark}[Robust Constraint Satisfaction]
\label{rem:constraint-sat}
By the RPI property of $\cZ_i^\infty$, the tube constraint gives
$e_i^{0|t} := x_i^t - s_i^{*,0|t} \in \cZ_i^\infty$ for all $t$, so the
private constraints hold as
$x_i^t \in \cS_i \oplus \cZ_i^\infty \subseteq \cX_i$ and
$u_i^t \in \cV_i \oplus K_i^p \cZ_i^\infty \subseteq \cU_i$, while summing over
neighbors yields
$\sum_{j} F_j x_j^t + G_j u_j^t \le \bar b_i + \sum_{j} d_j = b_i$;
hence, the original constraints \eqref{eq:ug-priv} and \eqref{eq:ug-shared}
are satisfied for every realization $w_i^t \in \cW_i$.
\end{remark}

\section{Conclusion}\label{sec:conclusion}

We proposed a robust receding horizon game framework for multi-agent
systems with additive disturbances and shared constraints. The framework
combines tube-based constraint tightening, equipped with a DARE-based terminal cost and decoupled positively invariant terminal sets. These ingredients guarantee recursive feasibility for every bounded disturbance realization. Exploiting the potential-game structure of the GNEP, we further established asymptotic convergence of the nominal states to the steady-state vGNE and showed that each agent's actual state converges to a neighborhood of the vGNE determined by the minimal RPI set.

There remain several interesting directions for future research. Analyzing closed-loop stability under general economic costs and when the cost functions of the agents are coupled remains a challenging open problem. Similarly, incorporating chance constraints or risk measures in the uncertain constraints is important to reduce conservatism of the robust approach studied here. Finally, developing distributed algorithms for the above mentioned extensions is an important direction for future research. 
\bibliographystyle{IEEEtran}
\bibliography{rhg_dinesh}
\end{document}